\documentclass[12pt]{article}

\usepackage[utf8]{inputenc}

\usepackage{amssymb}
\usepackage{latexsym}
\usepackage{exscale}

\usepackage[dvipsnames]{xcolor}

\usepackage{tikz}
\usetikzlibrary{trees}

\usepackage{amssymb, amsmath}
\usepackage{enumerate}
\usepackage{enumerate}
\usepackage[top=1in,left=1in,right=1in,bottom=1in]{geometry}
         \usepackage{ulem}
          \usepackage{xspace}
          \usepackage{cancel}

\usepackage{amssymb, amsmath}
\usepackage{enumerate}
\usepackage{enumerate}
\usepackage[top=1in,left=1in,right=1in,bottom=1in]{geometry}
\usepackage{ulem}
   \newcommand{\jg}[1]{{\color{Maroon}{#1}}}
   \newcommand{\jb}[1]{{\color{blue}{#1}}}

\usepackage{verbatim}

\def \beq{\begin{equation}}
\def \eeq{\end{equation}}

\newtheorem{remark}{Remark}

\renewcommand{\rq}[1]{(\ref{#1})}
\newtheorem{lemma}{Lemma}
\newtheorem{prop}{Proposition}
\newtheorem{thm}{Theorem}

\newcommand{\bR}{{ \mathbb R  }}

\newcommand{\bZ}{\Bbb Z}

\newcommand{\bN}{\mathbb{N}}
\newcommand{\bK}{\Bbb K}

\newcommand{\la}{\mbox{$\lambda$}}

\newcommand{\ra}{\mbox{$\mapsto $}}

\newcommand{\pa }{\partial }
\newcommand{\f}{\varphi}

\newcommand{\ep}{\epsilon}

\newcommand{\al}{\alpha }

\newcommand{\La}{\Lambda }

\newcommand{\om}{{\omega}}

\def\<{\langle} \def\>{\rangle}

\newcommand{\MM}{{\mathcal M}}

\newcommand{\A}{\tilde{A_D^{\alpha}}}
\newcommand{\AN}{\tilde{A_N^{\alpha}}}
\newcommand\CM[1]{\vskip3mm\fbox{\parbox{5in}{#1}}\vskip3mm}

\title{Null-controllability for the beam equation with
structural damping. Part 2: Integration by parts for fractional Laplacians and boundary control.}
\author{Sergei Avdonin, and  Julian Edward\footnote{corresponding author} }
\date{June 2026}
\begin{document}

\maketitle

Sergei Avdonin: Department of Mathematics and Statistics, University of Alaska Fairbanks, Fairbanks, 99775, Alaska, USA, saavdonin@alaska.edu
          
Julian Edward:  Department of Mathematics and Statistics, Florida International University, Miami, FLorida 33199, USA, edwardj@fiu.edu

\vspace{9in}

\maketitle
\begin{abstract}
Let $\Delta$ be the Neumann Laplacian on the interval $(0,\pi)$, and let $T>0$. An integration by parts formula is proven for the spectral fractional Laplacian, $(-\Delta)^\al$, for $\al \in (0,1)$. 
As an application, we prove well-posedness results for the structurally damped beam equation
$$u_{tt}+\Delta^2 u+\rho (-\Delta)^\al u_t=0, x\in (0,\pi),t>0$$
with various boundary conditions including
$$
u_x(0,t)=u_{xxx}(0,t)=0;\ u_x(\pi,t)=f(t),\ u_{xxx}(\pi,t)=0,
$$
and $f\in L^2(0,T)$
and appropriate initial conditions. 
Viewing $f$ as a control, we prove null-controllability.

Analagous results are proven for higher order controls, and for the Dirichlet Laplacian.
\end{abstract}


\section{Introduction}

Let $\Delta$ be the  Laplacian: $\Delta =\pa_x^2$ on the interval $(0,\pi)$. 
We will denote the Dirichlet Laplacian by $A_D=-\Delta$ 
with operator domain either $H^2(0,\pi)\cap H^1_0(0,\pi)$, 
and the Neumann Laplacian by $A_N=-\Delta$ with operator domain 
$\{ u\in H^2(0,\pi ): u'(0)=u'(\pi)=0\}.$
It is well known that $A_D,A_N$ are self-adjoint with non-negative spectra, and hence $A_D^{\alpha},A_N^\al$ are defined by the Spectral Theorem for all  $\alpha>0$. In what follows, we will denote by $A_j$
the operator $A_D$ or $A_N$ when there is no need for distinction. Denote by $X_j^p$ the operator domain of $A_j^{p/2}$.

We will study boundary control problems for the equation
\beq 
u_{tt}+A_j^2 u+\rho A_j^\al u_t=0,\ x\in (0,\pi), \ t>0, \ j= \mbox{D or N.}\label{beam0}
\eeq
Here $\al, \rho$ are positive constants.
The classical case, $\al =1$, was studied in \cite{AE}, but it's natural to also consider $\al\in [0,1).$
 Throughout this paper, controllability will always mean
the ability of steering any initial state $(u(x,0),u_t(x,0))$ to zero over a finite time by some appropriate
input function $f$ (i.e. exact controllability to zero or null controllability).
The boundary controls will be either 
\begin{eqnarray}
u(0,t)=u_{xx}(0,t) & = & 0,\nonumber \\
u(\pi,t) & = & f(t),\nonumber\\ 
u_{xx}(\pi,t) & = & g(t),\label{dir}
\end{eqnarray}
for the Dirichlet Laplacian, or
\begin{eqnarray}
u_x(0,t)=u_{xxx}(0,t) & = & 0,\nonumber\\
u_x(\pi,t) & = & f(t),\nonumber\\ 
u_{xxx}(\pi,t) & = & g(t),\label{neu}\\
\end{eqnarray}
for the Neumann Laplacian. 

The term $A_j^\al u_t$  models a specific dissipative effect, known as structural damping, when $\al \in (0,2)$. To the best of our knowledge, this was introduced in \cite{CR} assuming $\al =1$:
“The basic property of structural damping, which is said to be consistent with empirical studies,
is that the amplitudes of the normal modes of vibration are attenuated at rates which are proportional to the oscillation frequencies.” This model was also studied under the name “proportional
damping” (cf. \cite{Ba}). The quite different case $\al=2$ is known as “Kelvin–Voigt” damping. When
the right hand side of \rq{beam0} is replaced by $f(x,t)$ and $\al \in (0,2]$, this is the first class of parabolic-like control models considered
in \cite{LT}, see also \cite{AL}. 

For non-integer values of $\al$, and particularly for $f$ non-zero, the well-posedness of initial boundary value problems with either \rq{dir} or \rq{neu}  seems until recently to be unknown. A significant obstruction has been the lack of an integration by parts formula for $A_j^\al$, $j=$D or N, when applied to functions that don't satisfy homogeneous boundary conditions.  Recently, such a formula was proven for the Dirichlet Laplacian, \cite{E2} and \cite{AE1}. In this paper we prove such an integration by parts formula for $A_N^\al$. Perhaps partly  because of the lack of such a formula, many alternative, non-equivalent definitions of 
 $(-\Delta)^{\al}$  have been proposed, see \cite{G}, \cite{Di}, \cite{APR}, \cite{L} and references therein.

\subsection{Main Results}

In applications to boundary value problems, it is natural to consider extensions of  $A^{\al}$ to the case of non-homogeneous boundary conditions, and to use integration by parts  to define a weak solution. 

The extension we consider is based on the following formula, which can be found in \cite{SV} and was used for $A_D^\al$ in \cite{E2}, \cite{AE1}.
Let $p_j(t,x,y)$ be the heat kernel associated to the $A_j$ on $(0,\pi)$. 
The following representation for the Dirichlet Laplacian can be found in \cite{SV} for $\al <1$ and $w\in C_0^\infty(0,\pi)$: 
$$
A_D^{\al}w(x)=PV \int_{0}^{\pi}(w(x)-w(y))J_D(x,y)dy+ \kappa(x)w(x).
$$
Here PV stands for principle value, and 
\beq 
J_j(x,y)=\frac{\al}{\Gamma(1-\al)}\int_0^\infty p_j(t,x,y)t^{-\al -1}dt, \ \kappa(x) =\frac{\al}{\Gamma(1-\al)}\int_0^\infty \big (1-\int_0^\pi p_D(t,x,y)dy\big )t^{-\al -1}dt.\label{SV}
\eeq
For the Neumann Laplacian, the following analogue will be proven in the appendix: 
\beq 
A_N^{\al}w(x)=PV \int_{0}^{\pi}(w(x)-w(y))J_N(x,y)dy,
\label{SVN}
\eeq
with $J_N$ given by \rq{SV}. The proof of this result follows an argument used in \cite{AD} for the Dirichlet Laplacian, but also see \cite{CS}. We use \rq{SV} and \rq{SVN} to extend the fractional Laplacian to functions satisfying non-homogeneous boundary conditions: 

{\bf Definition}
 Suppose $\al \in (0,1).$ For $j=D$ or $N$,
the extended spectral fractional Laplacian, $\tilde{A}^\al_j$, is defined by \rq{SV} for $j=D$, and by \rq{SVN} for $j=N$.

\



For $\al \in (0,1)$, we set the operator domain of $\tilde{A}^\al_j$ as the set 
of measurable functions $w$ such that $\tilde{A}^\al_j w\in L^1(0,\pi )$.  For more on the properties of 
$\tilde{A}^\al_j$, the interested reader is refered to \cite{AD} and \cite{CS}.
It is shown in (\cite{AD}, Lemma 38) that for $\al \in (0,1), \ \phi \in X^\al_D$, we have
\beq
\tilde{A}^\al_D\phi =A^\al \phi.\label{aa}
\eeq

The following result was proven in \cite{E2}.
\begin{thm}\label{ibpd}

 Assume 
$$\al \in (0,1).$$
Let $u\in H^3(0,\pi) $.
Then $\A u\in C(0,\pi)\cap L^1(0,\pi),$
and
\beq
\int_{0}^{\pi}\A u(x)\ w(x)\ dx =\int_{0}^{\pi}u(x)\ A_D^{\al}w(x)\ dx, \ \forall w\in X^3_N.\label{symma2}
\eeq

\end{thm}

\begin{remark}
The assumption  that $u,w\in H^3(0,\pi)$ guarantees that $u,w\in C^2([0,\pi])$, which appears to be needed in our proof.  
\end{remark}


One of the main results of this paper is an extension of Theorem \ref{ibpd} to the Neumann Laplacian.
\begin{thm}\label{ibpn}
\ 

 Assume $\al \in (0,1).$
Let $u\in H^3(0,\pi), \ w\in X_N^3 $
Then $\AN u\in C(0,\pi)\cap L^1(0,\pi)$, and
$$\int_{0}^{\pi}\AN u(x)\ w(x)\ dx =\int_{0}^{\pi}u(x)A^\al_N w(x)\ dx.
$$

\end{thm}
Key to the analysis of $\tilde{A}^\al_N$ is the following formula for the heat kernel of the Neumann Laplacian:
\beq 
p_N(t,x,y)= \frac{1}{2\sqrt{\pi t}}\sum_{n=-\infty}^\infty 
(e^{-(x-y+2n\pi)^2/(4t)}+e^{-(x+y+2n\pi)^2/(4t)}),\label{pN}
\eeq 
which will allow us to decompose $J_N$ as an infinite sum of explicit terms. We suspect this formula is known in some communities, but we could not find it in the literature. The counterpart for the Dirichlet Laplacian, used in \cite{E2}, appears in \cite{BST}.

Using Theorem \ref{ibpd}, in \cite{E2} we 
 defined a weak solution to, and discussed well-posedness for,  the initial boundary value problem (IBVP)
\begin{eqnarray}
u_{tt}+A_D^2 u+\rho \A u_t& = & 0\label{beam}\\
u(0,t)=u_{xx}(0,t) & = & 0\label{bc1}\\
u(\pi,t) & = & f(t)\label{bc2}\\ 
u_{xx}(\pi,t) & = & g(t)\label{bc3}\\
u(x,0)=u^0(x),\ u_t(x,0)& = & u^1(x),\label{init}
\end{eqnarray}
with a positive constant $\rho$, and  $\al\in [0,1)$.

We now consider the Neumann Laplacian. 
Using Theorem \ref{ibpn}, we define a weak solution to the following IBVP:
\begin{eqnarray}
u_{tt}+A_N^2 u+\rho \AN u_t& = & 0,\label{beamn2}\\
u_x(0,t)=u_{xxx}(0,t) & = & 0,\label{bc1n2}\\
u_x(\pi,t) & = & f(t),\label{bc2n2}\\ 
u_{xxx}(\pi,t) & = & g(t),\label{bc3n2}\\
u(x,0)=u^0(x),\ u_t(x,0)& = & u^1(x).\label{initn2}
\end{eqnarray}

\begin{thm}\label{wellposedN}
Let $T>0$. Suppose $u^0=u^1=0$.

A) (Neumann control)
Let $g=0$, and  $f\in L^2(0,T).$  Let $\al \in (0,1).$
Then there exists a unique 
 solution $u^f$ to the \rq{beamn2}-\rq{initn2}, and the 
 mapping $f\mapsto u^f$ is continuous from $L^2(0,T)$ to $C(0,T;X^0_N)\cap C^1(0,T;X^{-2}_N).$

B) (Shear control)
Let $f=0$, and  $g\in L^2(0,T).$ 
 Suppose $\al \in (0,1]$. Then there exists a unique 
 solution $u^g$ to the system \rq{beamn2}-\rq{initn2}, and  the 
 mapping $g\mapsto u^g$ is  continuous from $L^2(0,T)$ to $C(0,T;X^{2}_N)\cap C^1(0,T;X^{0}_N).$

\end{thm}
Part A of this theorem is notably different from what holds classical structural damping, $\al =1$, which was studied in \cite{AE}, also see \cite{T1}. In that case, a continuous beam velocity requires a Neumann control more regular than $L^2$. The same distinction holds for Dirichlet control. This, and other differences with the classical case, will be discussed in the next subsection.

Having established well-posedness, it is not hard to prove certain null-control results. We state two here, and others in Section \ref{five}. 
\begin{thm}\label{bcD}
(Dirichlet null-controllability)

\ 

 Assume $\al\in (0,1)$ and $\rho < 2.$
Let $T>0.$ 
Given $(u^0,u^1)\in X^{-1}_D(0,\pi)\times X^{-3}_D(0,\pi)$, there exists $f\in L^2(0,T)$ such that the solution $u^f$ to the system  \rq{beam}-\rq{init}, with $g=0$, solves
$$u(x,T)=u_t(x,T)=0,
$$
and there exists a constant $C$ depending only on $\al,\rho ,T$ such that
$$\| f\|_{L^2(0,T)}\leq C(\| u_0\|_{X^{-3}_D}+\|u_1\|_{X^{-1}_D}).$$
\end{thm}

For Neumann null-controllability, we are required to consider quotient spaces, as noted in \cite{AE}. Thus
 we restrict $A_N$ to the subspace orthogonal to the constant functions, with associated Sobolev spaces 
$$X^p_{N,0}=\{ v \in X^p_N: \int_0^\pi v(x)\, dx=0\}.$$
\begin{thm}\label{bcN}
Assume  $\al <1$, $\rho < 2.$
Let $T>0.$ 
Given $(u^0,u^1)\in X^{0}_{N,0}\times X^{-2}_{N,0}$, there exists $f\in L^2(0,T)$ such that the solution $u$ to the system  \rq{beamn2}-\rq{initn2} solves
$$u(x,T)=u_t(x,T)=0,
$$
and there exists a constant $C$ depending only on $\al,\rho ,T$ such that
$$\| f\|_{L^2(0,T)}\leq C(\| u_0\|_{X^{0}_{N,0}}+\|u_1\|_{X^{-2}_{N,0}}).$$
\end{thm}

These theorems are proven using an observability inequality that follows from estimates found in \cite{AIS}, and extended in Section 2.4 here.

This paper is organized as follows. In the next subsection, we compare our results with the literature. In Section 2.1, we discuss the spectral solution of the uncontrolled system, discussing how  $\la, \rho$ determine the separation properties of the frequencies.  In Section 2.2, we present ``complex windows" type estimate which is a mild extension of the one found in \cite{AIS}, which will be used to prove our control results for $\rho <2$. 
In Section 3, we prove  Theorem \ref{ibpn}.  Then in Section 4, we prove our well-posedness results. In Section 5, we prove our null-controllability results.  In Section 6 we give concluding remarks about open questions and possible extensions. Finally, in the appendix we prove \rq{SVN} and a lemma showing that the exponential families associated to our problem form Bessel sequences.

\subsection{Literature review}

Regarding Theorems \ref{ibpd}, \ref{ibpn},
we are unaware of any such results in the literature aside from \cite{E2}. The Dirichlet,  Neumann, and Poisson problems for  the spectral fractional
Laplacian, $A_j^{\al}$ for $\al \in (0,1)$, $j=D, N$, have been studied in a number of works, see \cite{CS}, \cite{G}, \cite{AD},  and the references therein.

Recently, there has been extensive work studying other (non-spectral) definitions of the fractional Laplacian, see \cite{Di} and also \cite{BWZ} where the differences in controllability between the spectral fractional Laplacian and the integral-fractional Laplacian are discussed for heat-like equations. With many of these other definitions, an integration by parts formula is known, see for instance \cite{G3}, \cite{L}. We also note an extension of $A^\al_j$, not equivalent to ours, is presented and analysed in \cite{APR}; a brief comparison between this operator and ours is presented in \cite{E2}.

Integration by parts is often an essential ingredient in formulating the definition of weak solutions. Perhaps as a consequence, there seem not to be many results pertaining to well-posedness for the structurally damped beam equation for $\al$ non-integer, particularly for Dirichlet or Neumann controls. For moment control, $u_{xx}(\pi, t)=g(t)$,  with $\al \in (0,1/2)$, where Theorem \ref{ibpd} is not needed, regularity is discussed in \cite{H2}. 
There, it is shown that if $g\in L^2(0,T)$, the solution $u^g$ to the system \rq{beam}-\rq{init} will satisfy the following. Let $\hat{\al}=\min (-1/4,\al -1/2)$, then 
$$
u\in C(0,T;X_D^{\hat{\al}+1})\cap C^1(0,T;X_D^{\hat{\al}}),
$$
a slightly more regular set of spaces than those stated in our Theorem \ref{wpM}. To prove this, the author uses the ``operator Carleson measure criterion".

The well-posedness and regularity for $\al=1$
is studied in Triggiani's work in \cite{T1}, also see \cite{T}. 
The first work discusses both Dirichlet and  Neumann control, and presents results for domains in $N$ dimensions (but does not discuss null controllability), using properties of analytic semigroups.
It might be possible to refine our energy spaces using these methods.

Of course, when $\al =1$, classical integration by parts has boundary terms, unlike in Theorems \ref{ibpd} and \ref{ibpn}. The boundary term complicates the definition of the weak solution. For both Dirichlet and Neumann control, if $f\in L^2(0,T)$, then $u_t$ is not necessarily continuous in $t$, making it impossible to discuss exact controllability in the usual way. For this reason, in \cite{AE}, which uses Triggiani's framework,  we assume $f\in H_0^2(0,T)$.

We now discuss a second significant difference between the results in \cite{T1}, \cite{AE}  and those in this paper.
Adopting the notation of \cite{AE}, 
as a first step to solving the IBVP  
for  Neumann boundary conditions, one performs a harmonic lifting. More precisely, let
$u(x,t)=xf(t)+v(x,t)$. Then $v$ will satisfy
a structurally damped beam equation with internal forcing term but homogeneous boundary conditions.  However, this trick requires original IBVP to have boundary conditions
$$u_x(0,t)=-f(t), u_x(\pi ,t)=f(t),$$ which is less general than the boundary conditions we consider here, see \rq{neu}.
The same restriction arises in \cite{T1} in the  study of the case non-homogeneous boundary values for $u_{xxx}$. It is unclear to us if this restriction is just an artifact of the harmonic lifting.

This paper is a continuation of a series that includes \cite{AEI}, \cite{AE}, \cite{AE1}.
For null-controllability with Dirichlet or Neumann boundary  control on a beam, with $\al =1$,  the only work we are aware of is \cite{AE}.  In \cite{AE1}, the plate equation with boundary control either Dirichlet or $\Delta u|_{\pa \Omega}=g$ is considered, and null-controllability results are proven for both $\al =1$ and $\al <1$.

For $\al =1$, there are several other papers that prove boundary null-controllability results for structurally damped plate and beam equations.  
In \cite{mil},  the author considered boundary control in the case $\al =1,$ with $\rho <2.$
The boundary control, when restricted to the one dimensional setting, would be equivalent to 
$$u(0,t)=u_{xx}(0,t)=u(\pi,t)=0, u_{xx}(\pi,t)=g(t).$$
He proved null-controllability using a transmutation argument applied to the wave equation, along with results from \cite{AIS}.
It is unclear how easy it would be to ease the restriction $\al=1$ in his setting, since 
his argument uses integration by parts.
For null-controllability for a clamped beam with $\al =1$, see \cite{AIS}, and for other variants, see \cite{CR}.
The papers \cite{H},
  \cite{AIS}, consider boundary controllability in higher dimensional settings, but both assume $\al=1,\rho <2.$ Also see \cite{R86}.

For results on interior control, where the restriction $\al=1$ is easier to relax, the reader is referred to \cite{LT},\cite{AL},  \cite{mil}, and finally \cite{AEI}. In \cite{mit}, controllability is proven via a Carleman estimate, assuming $\al=1$.
 
.

\section{Frequency set and biorthogonal functions}\label{spectr}
\subsection{Frequency set for Dirichlet Laplacian.}\label{2.1}

Consider the eigenvalue problem
$$
A_D\, \f(x)  = \mu\, \f(x), \ x \in (0,\pi), \ \
\f (0) =\f(\pi)= 0.\\
$$
Clearly an orthonormal  basis of eigenfunctions is $\{ \f_n; \ n\in\bN\} $, 
$\f_n(x)=\sqrt{\frac{2}{\pi}}\sin ( n x)$, with corresponding eigenvalues 
$n^2.$
Obviously the $A_D^\al$ has the same eigenfunctions, with corresponding eigenvalues $n^{2\al}$.

Consider the IBVP :
\begin{eqnarray} 
w_{tt}+A^2 w+\rho A^\al w_t & = & 0,  \label{w_prime1} \\
w(0,t) =w(\pi ,t)=w_{xx}(0,t)=w_{xx}(\pi ,t)& = & 0, \\
w(x,0) & =& w^0(x),  \\
w_t(x,0) & = & w^1(x).  \label{w_prime4}
\end{eqnarray}


Set $w=\sum_{n=1}^\infty a_n(t)\f_n(x)$. Then 
\rq{w_prime1} implies
$$
\sum (a_n''+\rho a_n'n^{2\al}+a_nn^4   )\f_n(x)=0,
$$
hence
\beq
a_n''+\rho n^{2\al}a_n'+n^4a_n=0,\ \ \forall n\in \bN.
\eeq
Solving  $\la^2+\rho n^{2\al}\la+n^4=0$, we get
\beq 
\la=
\frac{-\rho n^{2\al}\pm \sqrt{\rho^2n^{4\al}-4n^4}}{2} =: \la_n^{\pm}.\label{pm}
\eeq 
Thus, if $\la_n^+\neq \la_n^-$, i.e., $\rho \ne 2n ^{2(1-\al)}$, for all $n$,
\beq
w(x,t)=\sum_1^\infty (c_n^+e^{\la_n^+ t}+c_n^-e^{\la_n^- t})\f_n(x),\label{du}
\eeq
with coefficients $c_n^\pm$ determined by the initial conditions.
If for some $n$ we have $\la_n^+=\la_n^-$, 
then we change the corresponding term in $w$ to
$$
(c_n^+e^{\la_n^+ t}+c_n^-te^{\la_n^+ t})\f_n(x).
$$

We now examine the gap properties of 
$\{ \la_n^{\pm}\}$. 
 In what follows we will use the \textit{frequency sets}
$$
\Lambda^+=\{\la_n^+, \ n\in \bN \}, \ \Lambda^- = \{\la_n^-, \ n\in \bN \},\ \Lambda=\La^+\cup \La^-.
$$
 and the modified frequency sequence
\beq \label{La}
\Lambda^\jg{i}=\{\la_k\}_{k\in\bK},\ 
\bK=\bZ\backslash \{0\},\ 
		\la_k =\left \{
		\begin{array}{cc}
			-i\la_k^+, & k>0,\\
			-i\la_{|k|}^-, & k<0.
		\end{array}
		\right .
\eeq	
Thus it is easy to see $\Lambda^i\subset\mathbb C^+$. We will refer to the subsets $\La^+,\La^-$ as ``branches" of $\La$.
We say that a discrete set $\MM$ is ``separable" if  $ \inf\,|\mu - \mu'|>0$ for all $\mu, \mu' \in \MM,\; \mu \neq \mu'. $ We will study this property for the sets $\La^+,\, \La^-$ and $\La.$

The following lemma collects some results proven in  \cite{AEI}.
\begin{lemma}\label{separat}

\ 

(1) For  $\al \in [0,1]$ and $\rho <2,$
 the frequency set   
 $\La$ is separable,

(2)  For $\rho =2$ and any $\al\geq 0$, we have
$\la_1^+=\la_1^-.$

(3) For any $\al \in (0,2)$ and $\rho >0$, the maximum possible multiplicity for each 
$\Lambda^+$ and $\Lambda^-$
is two.

\end{lemma}

\begin{lemma}\label{L2}
Let $T>0$, $\alpha \in (0,1),$ and $\rho>0$. 

A) Assume $\la_n^+\neq \la_n^-$ for all $n$. Then the family $\{ e^{t \la_n^+},e^{t \la_n^-}: n\in \bN \}$ forms a Bessel sequence in $L^2(0,T)$, 
i.e. for any $N$, there exists a positive constant $C$ independent of $N$ such that
$$
\int_0^T|\sum_{n=1}^N c_n^+e^{\la_n^+t}+c_n^-e^{\la^-_nt}|^2 \ dt\leq C \sum_{n=1}^N |c_n^+|^2+|c_n^-|^2.
$$

B) Suppose $\la_n^+=\la_n^-$ only for a finite  collection $\{n_i, \ i\subset I\}$.
Then the family of functions
$$
\{ e^{t \la_n^+},e^{t \la_n^-}: n\notin I \} \cup \{ e^{t \la_n^+},te^{t \la_n^+}: n\in I \} 
$$
forms a Bessel sequence in $L^2(0,T)$.
\end{lemma}
The proof of this lemma is deferred until the appendix.

\subsection{Frequency set for Neumann Laplacian}
Recall the eigenvalues of $A_N$ are
$\{ n^2: \ n\in \bN_0\}$, with corresponding normalized eigenfunctions being 
$$\{ 1/\sqrt{\pi}\}\cup \{ \frac{1}{\sqrt{2\pi}}\cos (nx)  :\ n\in \bN \}.$$
Then the discussion of the previous subsection easily extends to this case, including Lemmas \ref{separat} and \ref{L2}; the only, minor, caveat is we must adjoin one new (double) frequency, which we can label $\la^+_0=0$ with associated exponential function $1$.
\subsection{Complex Window Estimate}
We present some estimates used in the proofs of our null-controllability results.

The following result extends Theorem 1 in \cite{AIS}, also see \cite{B}. Our statement will pertain to frequencies for either the Dirichlet Laplacian, in which case the index set $K$ below will equal $\bK$, or the Neumann Laplacian, in which case the index set $K$ equals $\bZ$ to account for the extra frequency $\la =0$ associated with the latter operator. 
\begin{thm}\label{win}

Let $\al \leq 1$ and $\rho <2$.       For any $T,T'>0$, the operator
$$C_{T,T'}: \sum_{n=0,1}\sum_k  c_{k,n}t^n e^{i\la_k t} \mapsto \{c_{k,n}e^{-T'\Im(\la_k)};n=0,1;k\in K\}$$ 
is bounded from
$L^2(0,T)$ to $l^2$ with
$$
        \| C_{T,T'}\| \leq C\,\Psi (T ,T'),
$$
where $\Psi
(T,T'):=  e^{2Q(T)} e^{1/T'}$ with $Q(T)$ a constant that depends on $T$.
\end{thm}
Proof: Set $\sigma_j=\Im (\la_j)$.
Let 
$f=\sum_{n=0,1;k\in \bK}c_{k,n}t^ne^{i\la_kt}$ be a finite sum. Denote by
$\{ g_{k,n}: k\in \bK, n=0,1\}$ the  biorthogonal family associated to 
$\{t^n e^{i\la_kt},
n=0,1,k\in \bK \}$, whose existence is proven in (Proposition 1 in \cite{AEI}). Then by same result, there exist constants $C_1,C_2,$ and $ Q(T)$ such that 
\beq
\| g_{m,j}\| \leq C_1e^{Q(T)+C_2(\sigma_j)^\varkappa} , \label{geq}
\eeq
with  $\varkappa =1/2$ for either $\al <1$ or $\al=1,\rho \leq 2$. 
Then $c_k^n=(f, g_{k,n})$, where $(*,*) $ is the
scalar product in $L^2(0,T)$. Denote the norms on $L^2(0,T)$, resp. $\ell^2$, by 
$\| *\|_2$, resp. $\| *\|_{\ell^2}$, and $\< *,*\>$ the inner product on $\ell^2$. Let $\| *\|$ be the operator norm for the appropriate Hilbert spaces. 
        Therefore,
\begin{eqnarray*}
\|C_{T,T'} f\|_{\ell^2}^2& =
        &\sum_{n,k} |c_{k,n}|^2e^{-2T' \sigma_k}\\
&=&\sum (f, g_{k,n})e^{-2T'\sigma_k}\bar c_{k,n}\\
&= &(f,\sum  c_{k,n} g_{k,n}e^{-2T'\sigma_k})\\
&\le & \|f\|_2  \ \|\sum  c_{k,n} g_{k,n}e^{-2T'\sigma_k}\|_2.
\end{eqnarray*}
Let us introduce the Gramian matrix $\Gamma$: $l^2 \mapsto l^2$, with entries
$$
\Gamma_{m,j;k,n}= (  g_{j,m}e^{-T'\sigma_j}, g_{k,n}e^{-T'\sigma_k}).
$$
Then, letting $C_{T,T'}f=\vec{c}=\{ e^{-\sigma_jT'}c_{k,n}\},$
\begin{eqnarray*}
\|\sum c_{k,n} g_{k,n} e^{-2T'\sigma_k}\|^2_2
& =& \sum_{m,j}c_{j,m}e^{-\sigma_jT'}\sum_{n,k}\Gamma_{m,j;k,n}\bar{c}_{k,n}e^{-\sigma_kT'}\\
&=&  \< \vec{c},\overline{\Gamma\bar{\vec{c}}}\>\\
& \leq & \|\Gamma\|  \|\vec{c}\|_{\ell^2}^2.
\end{eqnarray*}
Hence we obtain 
$$\|C_{T,T'} \|\le\|\Gamma \|^{1/2}.$$
By the Gershgorin Theorem, 
$$
\|\Gamma  \|\leq \sup_{m,j}\{ e^{-\sigma_jT'}\sum_{n,k}e^{-\sigma_kT'}( g_{m,j},g_{k,n})\}.
$$
Letting $C$ be various constants that can change from line to line, by \rq{geq},
\begin{eqnarray*}
\sum_{n,k}e^{-\sigma_kT'}|( g_{m,j},g_{k,n})|& \leq & 
Ce^{2Q(T)}e^{C_2(\sigma_j)^\varkappa}\sum_{n,k}e^{-\sigma_kT'+C_2(\sigma_k)^\varkappa}\\
& \leq & Ce^{2Q(T)}e^{C_2(\sigma_j)^\varkappa}e^{C/T'}.
\end{eqnarray*}
Hence
\begin{eqnarray*}
\|C_{T,T'} \|&\leq &Ce^{Q(T)}e^{C/T'}\sup_{m,j}\{ e^{-\sigma_jT'/2}e^{C_2(\sigma_j)^\varkappa/2}\}\\
& \leq & Ce^{Q(T)}e^{C/T'}.\Box
\end{eqnarray*}

\section{Proof of Theorem \ref{ibpn}.}

Recall $A_N=-\pa_x^2$ has operator domain  $\{ u\in H^2(0,\pi): \ u'(0)=u'(\pi )=0\}$. Denote $X_N^{p}=Dom(A^{p/2}),$ so that
$X^1=H^1(0,\pi)$, $X^0=L^2(0,\pi)$. 
In this section, we wish to extend $A_N^\al$ to functions that don't satisfy homogeneous boundary conditions, using 
 the integral representation of the operator in \rq{SVN}, which we recall for the reader's convenience.
Let $p(t,x,y)$ be the heat kernel associated to the Neumann Laplacian on $(0,\pi)$. 
Then for $\al <1$ and $w\in X^2$:
\beq
\tilde{A}_N^{\al}w(x)=PV \int_{0}^{\pi}(w(x)-w(y))J_N(x,y)dy.\label{def}
\eeq 
Here PV stands for principle value, and 
\beq
J_N(x,y)=\frac{\al}{\Gamma(1-\al)}\int_0^\infty p(t,x,y)t^{-\al -1}dt.\label{JN}
\eeq
We will need to make more precise the principle value definition of 
$\int_0^\pi (w(x)-w(y))J(x,y)dy.$ For $\ep>0$,
let 
$$U_\ep =\{ (x,y)\in (0,\pi)\times (0,\pi): |x-y|\geq \ep\},$$
and let $J_\ep(x,y)=\chi_{U_\ep}(x,y)J(x,y).$
Then we define 
$$PV \int_{0}^{\pi}(w(x)-w(y))J(x,y)dy=\lim_{\ep \to 0}
\int_{0}^{\pi}(w(x)-w(y))J_\ep(x,y)dy.$$

We now prove Theorem \ref{ibpn}, which we restate for the reader's convenience. 

\noindent{\bf Theorem \ref{ibpn}}
Assume 
$$\al \in (0,1).$$
Let $w\in X_N^3.$ 
Let $u\in H^3(0,\pi) $.
Then $\AN u\in C(0,\pi)$, $\AN uw \in L^1(0,\pi),$
and
\beq
\int_{0}^{\pi}\AN u(x)\ w(x)\ dx =\int_{0}^{\pi}u(x)\ A^{\al}_Nw(x)\ dx.\label{symma2a}
\eeq

Proof: For readability, in this proof we drop the subscript $N$ from $\tilde{A}^\al$ and the functions $p,J$.
We need to prove that the function
$$
x\mapsto PV \int_0^{\pi} (u(x)-u(y))J(x,y)dy
$$
is in $C(0,\pi)\cap L^1(0,\pi)$, and 
$$
\int_{0}^{\pi} w(x)\ PV\int_0^\pi (u(x)-u(y))J(x,y)dy\ dx =\int_{0}^{\pi}u(x)\  
PV\int_0^\pi (w(x)-w(y))J(x,y)dy dx,
$$
i.e.
\beq 
\int_{0}^{\pi} w(x)\lim_{\ep \to 0}\int_0^\pi (u(x)-u(y))J_\ep (x,y)dy\ dx =\int_{0}^{\pi}u(x)\  
\lim_{\ep\to 0}\int_0^\pi (w(x)-w(y))J_\ep(x,y)dy dx.\label{symmb}
\eeq
We will in fact prove this assuming $w\in H^3(0,\pi )$, rather than the stronger hypothesis $w\in X_N^3.$
Since 
 $p(t,x,y)=p(t,y,x)$, it follows that $J_\ep(x,y)=J_\ep(y,x),$ and hence 
we have 
$$
\int_{0}^{\pi} w(x)\int_0^\pi (u(x)-u(y))J_\ep (x,y)dy\ dx =\int_{0}^{\pi}u(x)\  
\int_0^\pi (w(x)-w(y))J_\ep(x,y)dy dx,\ \forall \ep >0.
$$
Thus \rq{symmb} reduces to the analysis question of whether one can commute the limit in $\ep$ with the integral in $x$.
 However,
the integrals do not converge absolutely, so we need a careful argument. 
 
Let 
$$
C(\al )=\frac{\al\Gamma(\al+1/2) }{\Gamma(1-\al)}\frac{4^\al}{(\pi)^{1/2}}.
$$
We now show 
$$
J(x,y)=C(\al )
\sum_{n\in \bZ}
\left (|x- y+n\pi |^{-2\al-1}+
|x+ y+n\pi |^{-2\al-1}
 \right ) .
$$
In fact, to compute the right hand side of \rq{JN}, we will use the formula for $p$, which can be deduced from symmetry:
\begin{eqnarray}
p(t,x,y)&=& \frac{1}{2\sqrt{\pi t}}\sum_{n=-\infty}^\infty 
(e^{-(x-y+2n\pi)^2/(4t)}+e^{-(x+y+2n\pi)^2/(4t)}).\label{moi}
\end{eqnarray}
We compute
\begin{eqnarray}
\frac{1}{4^{\al+1/2}}\int_0^\infty
e^{-(x\pm y+2n\pi )^2/(4t)}t^{-\al -3/2}dt
& = & 
|x\pm y+2\pi n|^{-2\al -1}
\int_0^\infty 
e^{-s}s^{\al -1/2}ds\nonumber \\
& = & \Gamma(\al +1/2)|x\pm y+2\pi n|^{-2\al -1}
\label{j2}
\end{eqnarray}

We now  present a decomposition $J(x,y)=\sum_{m=1}^3{\cal J}^m(x,y)$, where the functions ${\cal J}^m$ will be defined below. 
Then it suffices to prove that for each $m$,

{\bf A:} the mapping 
\beq 
x\mapsto PV\int_0^\pi (u(x)-u(y)){\cal J}^m(x,y)dy\mbox{ is in } C(0,\pi)\cap L^1(0,\pi),\label{A}
\eeq
and

{\bf B:} the following equation holds for each $m$
$$
\lim_{\ep \to 0}\int_{0}^{\pi} w(x) \int_0^\pi \chi_{U_\ep}(x,y)(u(x)-u(y)){\cal J}^m(x,y)dy\ dx $$
\beq 
=\int_{0}^{\pi} \lim_{\ep \to 0}w(x)\int_0^\pi \chi_{U_\ep}(x,y)(u(x)-u(y)){\cal J}^m(x,y)dy\ dx.\label{B}
\eeq
Below, let $C$ be various positive constants
independent of $x,y$. Denote by $S$ the closed square 
$$S=[0,\pi]\times [0,\pi].$$

We  now define ${\cal J}^3$: 
\begin{eqnarray*}
4^{-\al}\sqrt{\pi}{\cal J}^3(x,y)
& = & \sum_{n\neq 0}|x- y+2\pi n|^{-2\al -1}
\int_{|x- y+2\pi n|^2/4}^\infty 
e^{-s}s^{\al -1/2}ds\\
& & +\sum_{n\neq 0,1}|x+ y+2\pi n|^{-2\al -1}
\int_{|x+ y+2\pi n|^2/4}^\infty 
e^{-s}s^{\al -1/2}ds.
\end{eqnarray*}
Then for $n\neq 0,$ the function $|x- y+2\pi n|^{-2\al -1}$ is continuous on $S$, while for $n\neq 0,1$, the function $|x+ y+2\pi n|^{-2\al -1}$ is continuous on $S$. Thus the series on the right hand side is convergent absolutely, uniformly in $(x,y)\in S$, so ${\cal J}^3$ is a continuous function on $S$. Also, note that the absolute convergence justifies the change in order of summation. This proves \rq{A} for $m=3$. Since ${\cal J}^3(x,y)={\cal J}^3(y,x)$, \rq{B} follows easily for $m=3$.

We now define ${\cal J}^1$:
\begin{eqnarray*}
 4^{-\al}\sqrt{\pi}{\cal J}^1(x,y)
&=&  \Gamma(\al +1/2)|x- y|^{-2\al -1}.
\end{eqnarray*}
Obviously we have ${\cal J}^1(x,y)={\cal J}^1(y,x)$. To prove \rq{A} is where  the principal value formula is vital.
We have, by Taylor's Theorem,
\begin{eqnarray}
 \int_0^{\pi} (u(x)-u(y))\chi_{U_\ep}(x,y)|x-y|^{-1-2\al}dy
& = & u'(x)\int_0^{\pi} (x-y)\chi_{U_\ep}(x,y)|x-y|^{-1-2\al}dy\nonumber\\
&  & +\int_0^{\pi}u''(z) (x-y)^2\chi_{U_\ep}(x,y)|x-y|^{-1-2\al}dy\nonumber\\
&&\label{breakdown}
\end{eqnarray}
with $z$ between $x$ and $y$. Because 
$u''(z)$ is continuous for $x,y\in [0,\pi]$, and because $1-2\al>-1$, the last integral on the right hand side converges, uniformly in $x$, to a continuous function of $x$, as $\ep\to 0$. The first integral we compute directly.
Assume for the moment that $\al\neq 1/2.$
An exercise in calculus gives
$$
\int_0^{\pi} (x-y)\chi_{U_\ep}(x,y)|x-y|^{-1-2\al}dy=
\left \{
\begin{array}{cc}
\frac{1}{1-2\al}((\pi -x)^{1-2\al}+\ep^{1-2\al}), & x<\ep,\\
\frac{1}{1-2\al}((\pi -x)^{1-2\al}+x^{1-2\al}), & \ep\leq x\leq \pi-\ep,\\
\frac{1}{1-2\al}(\ep^{1-2\al}+x^{1-2\al}), & x>\pi-\ep.
\end{array}
\right .
$$
Thus, as $\ep \to 0$, this last integral converges in $L^1(0,\pi)$ to 
$\frac{1}{1-2\al}((\pi -x)^{1-2\al}+x^{1-2\al})$, and evidently this limit is continuous in $(0,\pi)$. Furthermore, if  $w\in H^2(0,\pi)$, then  
we deduce
$$\lim_{\ep \to 0}\int_0^\pi
w(x)u'(x)\int_0^{\pi} (x-y)\chi_{U_\ep}(x,y)|x-y|^{-1-2\al}dydx$$
$$
=\int_0^\pi\lim_{\ep \to 0}
w(x)u'(x)\int_0^{\pi} (x-y)\chi_{U_\ep}(x,y)|x-y|^{-1-2\al}dydx.
$$
 The case 
$\al=1/2$ and $w\in X^2$ is handled similarly.
In summary, we have proven 
$$x\mapsto PV\int_0^{\pi} (u(x)-u(y)){\cal J}^1(x,y)dy
$$
is continuous and integrable on $(0,\pi)$, and that 
$$
\lim_{\ep\to 0}
\int_0^{\pi} \int_0^{\pi}w(x) (u(x)-u(y))\chi_{U_\ep}(x,y){\cal J}^1(x,y)dydx
=\int_0^\pi (PV\int_0^{\pi}w(x) (u(x)-u(y)){\cal J}^1(x,y)dy)dx,
$$
so \rq{A},\rq{B} hold for $m=1$.

We now analyse the contribution made by 
\begin{eqnarray*}
 2\sqrt{\pi}{\cal J}^2(x,y):
&=&  \Gamma(\al +1/2)(|x+ y|^{-2\al -1} +|x+ y-2\pi|^{-2\al -1}).
\end{eqnarray*}
We now consider the first term. We have, by Taylor's Theorem,
\begin{eqnarray}
 \int_0^{\pi} (u(x)-u(y))\chi_{U_\ep}(x,y)|x+y|^{-1-2\al}dy
& = & u'(x)\int_0^{\pi} (x-y)\chi_{U_\ep}(x,y)|x+y|^{-1-2\al}dy\nonumber\\
&  & +\int_0^{\pi} u''(z)(x-y)^2\chi_{U_\ep}(x,y)|x+y|^{-1-2\al}dy,\nonumber
\end{eqnarray}
with $z$ between $x$ and $y$.
We now bound the first integral; the easier argument for the second integral  is left to the reader.
Assume  $\al \neq 1/2$; the proof for $\al =1/2$ is similar. For fixed $x>0$, 
 by direct calculation, 
$$
\int_0^{\pi} (x-y)|x+y|^{-2\al -1}dy=
\big ( \frac{1}{\al}+\frac{1}{-2\al +1}\big )x^{-2\al+1}
-\frac{(x+\pi)^{-2\al+1}}{-2\al +1}-\frac{x(x+\pi)^{-2\al}}{\al}:=I(x).
$$
Since $I(x)$ is  $L^1(0,\pi)$, we conclude that as $\ep \to 0$,  
$$
\int_0^{\pi} \chi_{U_\ep}(x,y)(x-y)|x+y|^{-2\al -1}dy 
$$
converges in $L^1(0,\pi)$ to $I(x).$
Thus 
$$
\lim_{\ep \to 0}\int_0^{\pi}w(x)u'(x)\int_0^{\pi} \chi_{U_\ep}(x,y)(x-y)|x+y|^{-2\al -1}dydx$$
$$=\int_0^{\pi}w(x)u'(x)\lim_{\ep \to 0}\int_0^{\pi} \chi_{U_\ep}(x,y)(x-y)|x+y|^{-2\al -1}dydx,
$$
as desired. Furthermore, the function $u'(x)I(x)$ is clearly continuous on $(0,\pi)$.
The argument for $\al =1/2$ is similar. The term involving $|x+y-2\pi|^{-1-2\al}$ is handled similarly.
 This proves 
$$
x\mapsto PV \int_0^{\pi} (u(x)-u(y)){\cal J}^2(x,y)dy
$$
is continuous and integrable on $(0,\pi)$. 
Thus we have proved \rq{A},\rq{B} for $m=2.$

This completes the proof of the theorem.

\section{Well-posedness for structionally damped beam}

\subsection{Non-homogeneous Dirichlet condition.}
For this subsection, we denote $X^p_D$ by $X^p$, and the pairing $X^p-X^{-p}$ by $\< *,*\>_{X^p,X^{-p}}$.

 We propose to study the structurally damped beam equation with non-homegeneous boundary conditions, which we write formally as
\begin{eqnarray}
u_{tt}+A_D^2 u+\rho \A u_t& = & 0\label{beam'}\\
u(0,t)=u_{xx}(0,t) & = & 0\label{bc1'}\\
u(\pi,t) & = & f(t)\label{bc2'}\\ 
u_{xx}(\pi,t) & = & g(t)\label{bc3'}\\
u(x,0)=u^0(x),\ u_t(x,0)& = & u^1(x).\label{init'}
\end{eqnarray}
This  subsection, we will set $g=0$.
The adjoint problem is 
\begin{eqnarray}
w_{tt}+A_D^2 w-\rho A_D^\al w_t& = & 0\label{dbeam'}\\
w(0,t)=w_{xx}(0,t) & = & 0\label{dbc1'}\\
w(\pi ,t)=w_{xx}(\pi ,t) & = & 0\label{dbc2'}\\ 
w(x,T)=w_0(x),\ w_t(x,T)& = & w_1(x).\label{dinit'}
\end{eqnarray}
We will see below that the associated observation will be $w_{xxx}(\pi,t)$.

Let $b\in \bR$.
Standard Fourier series arguments using the asymptotics of $\la_n$ show that the mapping $(w^0,w^1)\mapsto w(x,t)$ is bounded from $X^{2+b}\times X^b$ to 
$$C^0(0,T;X^{2+b})\cap C^1(0,T,X^b),
$$
and there exists a constant $C$ independent of $(w^0,w^1)$ such that
$$
\int_0^T|w_{xxx}(\pi,t)|^2dt \leq C(\| w^0\|_{X_3}^2+\| w^1\|^2_{X^1}), \ \forall (w^0,w^1)\in W^3\times X^1.
$$
For proof, the reader is referred to a
 more general version proven in \cite{E2}.

To define the weak solution for Dirichlet control, set $g=0$, 
 and let $u$, resp. $w$, solve
\rq{beam'}-\rq{init'}, resp. \rq{dbeam'}-\rq{dinit'}.
Assume for the moment that $u\in C^2(0,T;H^4(0,\pi)),w\in C^2(0,T;X^4)$, so that integration by permissible using Theorem \ref{ibpd}. Then integration by parts yields
\begin{eqnarray*}
0&=&\int_0^T\int_0^{\pi}(w_{tt}+A^2w-\rho A^{\al}w_t)u\ dxdt\nonumber\\
&=&\int_0^{\pi}[w_tu-wu_t-\rho A^{\al}wu]_0^T dx+\int_0^Tw_{xxx}(\pi,t)f(t)dt.
\end{eqnarray*}
 Then we define $u$ to be a solution to \rq{beam'}-\rq{init'} if 
\begin{eqnarray}
0&=&\< w_t(*,T),u(*,T)\>_{X^1,X^{-1}}-\rho <A^{\al}w(*,T),u(*,T) \>_{X^1,X^{-1}}\nonumber \\
& & -\< w(*,T),u_t(*,T)\>_{X^3,X^{-3}}+\int_0^Tw_{xxx}(\pi,t)f(t)dt, \ \forall (w_0,w_1)\in X^3\times X^1.
\label{weak1}
\end{eqnarray}
We state a result proven in greater generality in \cite{E2}; for proof we refer the interested reader to that work. 
\begin{thm}\label{wpD}
Let $T>0$. Let $\al< 1$, and $\rho >0$.
Let $f\in L^2(0,T).$ Suppose $u^0=u^1=0$. Then there exists a unique 
 solution $u^f$ to the system \rq{beam'}-\rq{init'} with the 
 mapping $f\mapsto u^f$ is  continuous from $L^2(0,T)$ to $C(0,T;X^{-1})\cap C^1(0,T;X^{-3}).$
\end{thm}
\begin{remark}
It is unclear to us how to define a solution to the system \rq{beam'}-\rq{init'} for $\al \in (1,2)$ for non-homogeneous Dirichlet boundary conditions.
\end{remark}

We now present an example that serves two purposes. First, it uses the important result from \cite{E2}  that if $u\in H^3(0,\pi)$, then 
as elements of $X^{-2\al}_D$, $\A u=A^\al u$. Consequently it is easy to compute the Fourier series of $u$: if $u=\sum u_n\f_n$, then $\A u=\sum n^{2\al}u_n\f_n$. 
The example also
 shows that the set of weak solutions $u^f$ to \rq{beam'}-\rq{init'} for which integration by parts is permissible is non-trivial.

{\bf Example}\

We adapt the harmonic lifting argument used in \cite{T1}, \cite{AE}, but with $\al <1$, to convert our problem into a problem with  homogeneous boundary conditions . 
	Let $U(x,t)=\frac{x}{\pi}f(t)$, and $v(x,t)=u(x,t)-U(x,t).$
	Then $v=v^f$ satisfies
	\begin{eqnarray}
		v_{tt}+A_D^2 v+\rho A_D^\al v_t & = & -\frac{x}{\pi}f''(t)-\frac{\rho} {\pi}f'(t)\A x,\label{v1}\\
		v(0,t) =v(\pi ,t)=v_{xx}(0,t)=v_{xx}(\pi ,t)& = & 0\\
		v(x,0) & =& u_0(x)-\frac{x}{\pi}f(0)\\
		v_t(x,0) & = & u_1(x)-\frac{x}{\pi}f'(0).\label{v4}
	\end{eqnarray}
Assume $f\in C^{\infty}(0,T)$ with $f(0)=f'(0)=0$, and $u^0=u^1=0.$ Evidently, $U\in C^\infty(0,T,H^4(0,\pi)).$

We now obtain a formula for $v$. First, we note the following Fourier expansion
$$
\frac{x}{\pi}=\sum_1 x_m\f_m,\ 
x_m=\sqrt{\frac{2}{\pi}}
\frac{(-1)^{m+1}}{m}.
$$ 
We now compute the Fourier series for $\A x$. First, observe that since $x\in X^0$, it follows from spectral theory that $A_D^\al x\in X^{-2\al}$. Since the function $u=x$ satisfies the hypothesis necessary in Theorem \ref{ibpd}, 
it was proven (\cite{E2}, Corollary 1) that, as objects in $X^{-2\al}$, we have $\A x=A_D^\al x$. Thus we have  
\begin{eqnarray*}
\< \A \frac{x}{\pi},\f_n\>_{X^{-2\al},X^{2\al}} & = & \< A^\al_D\frac{x}{\pi},\f_n\>_{X^{-2\al},X^{2\al}}\\
& = & \<  \frac{x}{\pi},A^\al_D \f_n\>_{X^{0},X^{0}}\\
& = & n^{2\alpha}x_n=: y_n.
\end{eqnarray*}

Set $v^f=\sum a_n(t)\f_n(x)$. Then 
\rq{v1} implies
$$
\sum (a_n''+\rho a_n'n^{2\al}+a_nn^4   )\f_n(x)=-\sum [f''(t)x_n+\rho f'(t)y_n]\f_n(x),
$$
and since $f\in C_c^\infty(0,T)$,
\beq
a_n''+\rho n^{2\al}a_n'+n^4a_n=-f''(t)x_n-\rho f'(t)y_n,\ a_n(0)=a_n'(t)=0,\ \forall n\in \bN.\label{an1}
\eeq
We set 
$$
\beta_n=\frac{-\rho n^{2\al}}{2}, \ 
\al_n=\frac{\sqrt{4n^4-\rho^2n^{4\al}}}{2},\mbox{ so }
\la_n^{\pm}=\beta_n\pm i\al_n.
$$
We solve \rq{an1} using 
variation of parameters. 
We have the  Wronskian equalling
$-2i\al_ne^{-2\beta_nt}$, 
hence
\begin{eqnarray*}
a_n(t)
& =& \frac{1}{2i\al_n}
\int_0^t\big ( f''(s)x_n+\rho f'(s)y_n\big )
(e^{\la_n^+(t-s)}-e^{\la_n^-(t-s)})ds,
\ n\in \bN .
\end{eqnarray*}
In integration by parts argument now shows that 
$$
v^f\in C^\infty (0,T;X^4)),
$$
and hence the same regularity will hold for $u^f=v^f+U$ as desired.

\subsection{Moment control}

We now discuss the well-posedness for moment control, so $f=0$ in 
\rq{beam'}-\rq{init'}. For this subsection, we denote $X^p_D$ by $X^p$, and the pairing $X^p-X^{-p}$ by $\< *,*\>_{X^p,X^{-p}}$.

In \cite{E2}, we show there exists a constant $C$ independent of $(w^0,w^1)$ such that
$$
\int_0^T|w_{x}(\pi,t)|^2dt \leq C(\| w^0\|_{X_1}^2+\| w^1\|^2_{X^{-1}}), \ \forall (w^0,w^1)\in X^3\times X^1. 
$$

Assume the solution, 
$u$, is sufficiently regular to permit application of Theorem \ref{ibpd}.
Then
integration by parts in this case gives
$$
0=\int_0^{\pi}[w_tu-wu_t-A_D^{\al}wu]_0^T dx-\int_0^Tw_{x}(\pi,t)g(t)dt.
$$
We define $u$ to be the weak solution to \rq{beam'}-\rq{init} if, for all $(w^0,w^1)\in X^1\times X^{-1}$, 
\beq
0=\< w^1-\Delta^\al w^0,u(*,T)\>_{X^{-1},X^{1}}-\< w^0,u_t(*,T)\>_{X^1,X^{-1}}-\int_0^T w_x(\pi ,t)g(t)dt .\label{wm0}
\eeq
We recall from \cite{E2}:
\begin{thm}\label{wpM}
Let $T>0$. Assume $f=0$.
Let $g\in L^2(0,T).$  Then there exists a unique 
 solution $u^g$ to \rq{beam'}-\rq{init} with the 
 mapping $g\mapsto u^g$ is  continuous from $L^2(0,T)$ to $C(0,T;X^{1})\cap C^1(0,T;X^{-1}).$
\end{thm}

\subsection{Non-homogeneous Neumann condition.}
Let $p\in \bR$. We define the norms on $X^p_N,X^{p}_{N,0}$. Let $\{ \f_n(x): \ n\geq 0\}$ be the normalized eigenfunctions of $A_N$. Then 
$$
\| \sum_{n=0}^\infty a_n\f_n\|^2_{X^p_N}=|a_0|^2+\sum_{n=1}^\infty n^{2p}|a_n|^2,
\ \ \
\| \sum_{n=1}^\infty a_n\f_n\|^2_{X^p_{N,0}}=\sum_{n=1}^\infty n^{2p}|a_n|^2.
$$

We first consider the system 
\begin{eqnarray}
u_{tt}+A_N^2 u+\rho \tilde{A}_N^\al u_t& = & 0\label{beamn}\\
u_x(0,t)=u_{xxx}(0,t) & = & 0\label{bc1n}\\
u_x(\pi,t) & = & f(t)\label{bc2n}\\ 
u_{xxx}(\pi,t) & = & g(t)\label{bc3n}\\
u(x,0)=u^0(x),\ u_t(x,0)& = & u^1(x)\label{initn}
\end{eqnarray}
with $g=0$.
We recall the adjoint problem:
\begin{eqnarray}
w_{tt}+A_N^2 w-\rho A_N^\al w_t& = & 0\label{dbeamn}\\
w_x(0,t)=w_{xxx}(0,t) & = & 0\label{dbc1n}\\
w_x(\pi ,t)=w_{xxx}(\pi ,t) & = & 0\label{dbc2n}\\ 
w(x,T)=w^0(x),\ w_t(x,T)& = & w^1(x).\label{dinitn}
\end{eqnarray}
We will see the associated observation will be $w_{xx}(\pi,t)$.
\begin{lemma}\label{adjnreg} Let $T>0$. For any $\rho >0$,
\

A)  The mapping $(w^0,w^1)\mapsto w$ is a continuous map 
$$X_N^{2+b}\times X_N^{b}\mapsto C(0,T;X_N^{2+b})\cap C^1(0,T; X_N^{b}).
$$ 

B) The mapping $(w^0,w^1)\mapsto w$ is a continuous map 
$$X_{N,0}^{2+b}\times X_{N,0}^{b}\mapsto C(0,T;X_{N,0}^{2+b})\cap C^1(0,T; X_{N,0}^{b}).
$$ 

C) There exists a constant independent of $(w^0,w^1)\in X^2_N\times X^0_N$ such that 
$$
\int_{t=0}^T|w_{xx}(\pi ,t)|^2dt\leq C(
\| w^0\|^2_{X^2_N}+\| w^1\|^2_{X^0_N}).
$$

\end{lemma}
We remark that part C also holds if $X^2_N\times X^0_N$ is replaced by $X^2_{N,0}\times X^0_{N,0}$.

Proof of part A: 
 Assume $(w^0,w^1)\in X_N^{2+b}\times X_N^{b}$.
Recall
$$
w(x,t)=c_0^++c_0^-(T-t)+\sum_1^\infty (c_n^+e^{\la_n^+(T-t) }+c_n^-e^{\la_n^-(T- t)})\f_n(x),
$$
with $c_0^+=w^0_0$, $c^-_0=-w^1_0$, and 
\begin{eqnarray}
c_n^+=&(-w_n^1-\la_n^-w_n^0)/\sqrt{\rho^2n^{4\al}-4n^4}=:& (-w_n^1-\la_n^-w_n^0)/q_n,\label{cn+1}\\
c_n^-=&(w_n^1+\la_n^+w_n^0)/\sqrt{\rho^2n^{4\al}-4n^4}
=:&(w_n^1+\la_n^+w_n^0)/q_n.
\label{cn-1}
\end{eqnarray}
We now  prove the regularity of $w_t(*,t)$. It follows from the results in Section \ref{2.1}
that $\{ 1,t, e^{\la_n^+ t},e^{\la_n^- t},n\in \bN \}$ forms a Bessel sequence in $L^2(0,T)$. Hence
\begin{eqnarray*}
\| w_t(*,t)\|_{X^b}^2 & = & |c_0^-|^2+
\sum_1^\infty n^{2b}|\la_n^+c_n^+e^{\la_n^+ t}+c_n^-\la_n^-e^{\la_n^- t}|^2\\
& = & |c_0^-|^2+\sum_1^\infty \frac{n^{2b}}{q_n^2}
|\la_n^+(-w_n^1-\la_n^-w^0_n)e^{\la_n^+t}+
\la_n^-(w_n^1+\la_n^+w^0_n)e^{\la_n^-t}|^2\\
& \leq &|w^1_0|^2+ C\sum_1^\infty \frac{n^{2b}}{q_n^2}
(|\la_n^-w_n^1|^2+|\la_n^+\la_n^-w^0_n|^2)\\
& \leq &|w^1_0|^2+ C\sum_1^\infty \frac{n^{2b}}{q_n^2}
(|n^{2}w_n^1|^2+|n^4w^0_n|^2)\\
& \leq &|w^1_0|^2+ C\sum_1^\infty {n^{2b}}
(|w_n^1|^2+|n^2w^0_n|^2) \\
& \leq & C (\| w^0\|_{X^{2+b}_N}^2+\| w^1\|^2_{X^b_N}).
\end{eqnarray*}
Here $C$ is some positive constant depending only on $\al, \rho$, and we used  the formula $\la_n^+\la_n^-=n^4$.
Because the series converges uniformly in $t$, we have
$w_t\in C(0,T;X^b).$
The fact that $w\in C(0,T;X^{2+b})$ can be proved similarly.

The similar proofs of part B,C is left to the reader. 
$\Box$

\

Let $u$, resp. $w$, solve
\rq{beamn}-\rq{initn}, resp. \rq{dbeamn}-\rq{dinitn}.
Assume for the moment that $u\in C^2(0,T;H^4),w\in C^2(0,T;X_N^6)$, so that integration by permissible using Proposition \ref{ibpn}. Then integration by parts yields
\begin{eqnarray*}
0&=&\int_0^T\int_0^{\pi}(w_{tt}+A_N^2w-\rho A_N^{\al}w_t)u\ dxdt\\
&=&\int_0^{\pi}[w_tu-wu_t-\rho A_N^{\al}wu]_0^T dx-\int_0^Tw_{xx}(\pi,t)f(t)dt.
\end{eqnarray*}
 Then we define $u$ to be a solution to \rq{beam'}-\rq{init'} if 
\begin{eqnarray}
0&=&\< w_t(*,T),u(*,T)\>_{X^0_N,X^{0}_N}-\rho \<A_N^{\al}w(*,T),u(*,T)\>_{X^0_N,X^{0}_N}
\nonumber \\
& & +\< w(*,T),u_t(*,T)\>_{X^2_N,X^{-2}_N}+\int_0^Tw_{xx}(\pi,t)f(t)dt, \ \forall (w^0,w^1)\in X^2_N\times X^0_N.\label{weakN1}
\end{eqnarray}

{\bf Proof of Theorem \ref{wellposedN}, Part A.}
We  use a standard duality argument: fix $T>0$ and  $f\in L^2(0,T).$ Then, by Lemma \ref{adjnreg} part C, the mapping 
$$
X^2_N\times X^0_N\ni (w^0,w^1)\mapsto \int_0^Tw_{xx}(\pi,t)f(t)dt
$$
is a bounded linear functional. Setting this functional equal to 
$$
(w^0,w^1)\mapsto -\< \big ( w_t(*,T)-\rho A^{\al}w(*,T)\big ),u(*,T)\>_{X^0_N,X^{0}_N}-
 \< w(*,T),u_t(*,T)\>_{X^2_N,X^{-2}_N}\ ,
$$
by duality there exists a unique pair $(u(*,T),u_t(*,T))\in X^{0}_N\times X^{-2}_N$ such that \rq{weakN1} holds. This proves existence and uniqueness. The continuity properties in $t$ can now be proved using Fourier series; the details are left to the reader.
$\Box$

We conclude this section by remarking that Theorem \ref{wellposedN}, part A is easily shown to still hold if 
$X^{0}_N\times X^{-2}_N$ is replaced by 
$X^{0}_{N,0}\times X^{-2}_{N,0}$.
We note the definition of weak solution in this case:
\begin{eqnarray}
0&=&\< w_t(*,T),u(*,T)\>_{X^0_{N,0},X^{0}_{N,0}}-\<A_N^{\al}w(*,T),u(*,T)\>_{X^0_{N,0},X^{0}_{N,0}}
\nonumber \\
& & +\< w(*,T),u_t(*,T)\>_{X^2_{N,0},X^{-2}_{N,0}}+\int_0^Tw_{xx}(\pi,t)f(t)dt,\nonumber\\
&&\ \forall (w^0,w^1)\in X^2_{N,0}\times X^0_{N,0}.
\label{weakN}
\end{eqnarray}

\subsection{Non-homogeneous $u_{xxx}(\pi,t)$}
In this subsection, we prove part B of Theorem \ref{wellposedN}.
Thus, we consider:
\begin{eqnarray}
u_{tt}+A_N^2 u+\rho \tilde{A}_N^\al u_t& = & 0\label{beamt}\\
u_x(0,t)=u_{xxx}(0,t) & = & 0\label{bc1t}\\
u_x(\pi,t) & = & 0\label{bc2t}\\ 
u_{xxx}(\pi,t) & = & g(t)\label{bc3t}\\
u(x,0)=u^0(x),\ u_t(x,0)& = & u^1(x),\label{initt}
\end{eqnarray}
whose adjoint problem is given by \rq{dbeamn}-\rq{dinitn}.

Integration by parts in this case gives
$$
0=\int_0^{\pi}[w_tu-wu_t-A_N^{\al}wu]_0^T dx+\int_0^Tw(\pi,t)g(t)dt.
$$
Setting $u^0=u^1=0$, we 
define $u$ to be the weak solution to \rq{beamt}-\rq{initt} if for all $(w^0,w^1)\in 
X^{0}_N\times X^{-2}_N$,
\beq 
0=\< w^1-A_N^\al w^0,u(*,T)\>_{X^{-2}_N,X^{2}_N}-\< w^0,u_t(*,T)\>_{X^0_N,X^{0}_N}+\int_0^T w(\pi ,t)g(t)dt.\label{weaks}
\eeq
The statement of part B follows now from the following lemma:
\begin{lemma}
Let $T>0.$ There exists a positive constant depending only on $\al,\rho $ such that the following holds: 
$$
\int_0^T|w(\pi,t)|^2dt\leq C(\| w^0\|_{X^0_N}^2+\| w^1\|_{X^{-2}_N}^2).
$$
\end{lemma}
Proof: 
 It follows from the results in Section \ref{2.1}
that $\{ 1,t, e^{\la_n^+ t},e^{\la_n^- t},n\in \bN \}$ forms a Bessel sequence in $L^2(0,T)$. It follows from this, and then by 
\rq{cn+1},\rq{cn-1}, that
\begin{eqnarray*}
\int_0^T|w(\pi ,t)|^2dt & = & \int_0^T|c_0^++c_0^-(T-t)+\sum_1^\infty(-1)^n(c_n^+e^{\la_n^+(T-t)}+c_n^-e^{\la_n^-(T-t)}|^2\ dt\\
& \leq & C\sum_{n=0}^\infty (|c_n^+|^2+|c_n^-|^2)\\
& = & C(\| w^0\|_{X^0}^2+\|w^1\|^2_{X^{-2}}).\Box
\end{eqnarray*}

\

\section{Null Controllability results}\label{five}
\subsection{Proof of Theorem \ref{bcD}}
For this subsection, for readability, we drop the subscripts and write $X^p_D$ as $X^p$.

\begin{eqnarray}
u_{tt}+A_D^2 u+\rho \A u_t& = & 0\label{beam9c}\\
u(0,t)=u_{xx}(0,t) & = & 0\label{bc9c}\\
u(\pi,t) & = & f(t)\label{bc19c}\\ 
u_{xx}(\pi,t) & = &0,\label{bc39c}\\
u(x,0)=u^0(x),\ u_t(x,0)& = & u^1(x).\label{init9c}
\end{eqnarray}

Let $w$ be the solution of the adjoint equation.
By \rq{weak1} we have 
\begin{eqnarray*}
0&=&[\< w_t(*,t),u(*,t)\>_{X^1,X^{-1}}-\<A^{\al}w(*,t),u(*,t) \>_{X^1,X^{-1}}]^T_0\nonumber \\
& & +[\< w(*,T),u_t(*,T)\>_{X^3,X^{-3}}]_0^T+\int_0^Tw_{xxx}(\pi,t)f(t)dt, \ \forall (w_0,w_1)\in X^3\times X^1, 
\label{weak1*}
\end{eqnarray*}
so $f$ being a null control would be equivalent to 
\begin{eqnarray}
0&=&-\< w_t(*,0),u(*,0)\>_{X^1,X^{-1}}+\<A^{\al}w(*,0),u(*,0) \>_{X^1,X^{-1}}\nonumber \\
& & -[\< w(*,0),u_t(*,0)\>_{X^3,X^{-3}}+\int_0^Tw_{xxx}(\pi,t)f(t)dt, \ \forall (w_0,w_1)\in X^3\times X^1. 
\label{weak1*c}
\end{eqnarray}
Recall, 
$$w(x,t)=\sum(c_n^+e^{\la_n^+(T-t)}+c_n^-e^{\la_n^-(T-t)})\f_n(x),
$$
so 
$$
w_{xxx}(\pi,t)=-\sum_n n^2\f_n'(\pi)(c_n^+e^{\la_n^+(T-t)}+c_n^-e^{\la_n^-(T-t)})
$$

By a standard duality argument applied to \rq{weak1}, null-controllability holds if and only if the following observability estimate holds
\beq \label{ob1a}
\| w_{xxx}(0,*)\|^2_{L^2(0,T)}\geq C(\| w(*,0)\|_{X^3}^2+\| w_t(*,0)\|_{X^1}^2), \forall (w^0,w^1)\in X^3\times X^1.
\eeq 
\begin{prop}\label{propd1}
Let $\al < 1, \rho <2.$ Then \rq{ob1a} holds. 
\end{prop}
\begin{remark}
 For $\al <1/2,$ a  result similar to Proposition \ref{propd1} was proven in \cite{E}, but in that paper it was not proven that observability implies null-controllability. 
\end{remark}
Proof of proposition: 

Recall $|\la^{\pm}_n|<Cn^2$ for large $n$. Also, the assumptions $\al< 1$ and $\rho <2$ imply $\Re (\la_n^\pm)=-\rho n^{2\al}/2.$
Thus we have
\begin{eqnarray}
\| w(*,0)\|_{X^3}^2+\| w_t(*,0)\|_{X^1}^2& = & \sum_1^\infty
n^6|c_n^+e^{\la_n^+T}+c_n^-e^{\la_n^-T}|^2
+n^2|\la_n^+c_n^+ e^{\la_n^+T}+ \la_n^-c_n^-e^{\la_n^-T}|^2\nonumber\\
& \leq &  C_1\sum_1^\infty
n^6
(|c_n^+|^2+|c_n^-|^2)\ |e^{\la_n^+T}|^2\nonumber\\
& \leq &  C_1\sum_1^\infty n^6
(|c_n^+|^2+|c_n^-|^2)e^{-\rho n^{2\al}T};\label{o1}
\end{eqnarray}
here, $C_1$ is a positive constant that depends only on $\rho, \al$.
We apply Theorem \ref{win}, choosing $T'=T/2.$
Thus
\begin{eqnarray}
\int_0^T|w_{xxx}(0 ,t)|^2dt & = &
\int_0^T|\sum_1^\infty n^2\f_n'(0)(c_n^+e^{\la_n^+(T-t)}+c_n^-e^{\la_n^-(T-t)})|^2dt\nonumber\\
& \geq & {C(T)}\sum_1^\infty |n^3c_n^+e^{\la_n^+T/2}|^2
+|n^3c_n^-e^{\la_n^-T/2}|^2\nonumber\\
& \geq & {C(T)}\sum_1^\infty n^6e^{-\rho n^{2\al}T/2}(|c_n^+|^2
+|c_n^-|^2).\label{o2}
\end{eqnarray}
Combining  \rq{o1} and \rq{o2}, the proposition follows. $\Box$

Theorem \ref{bcD} is proven.

\
\begin{remark}
It was shown in \cite{AEI} that
for  $\rho>2$, there might be multiplicities in the frequencies of the form $\la_n^+=\la_m^-$ with $m\neq n$; these are obstructions to null controllability. It is natural to ask if the controllability holds for  $\rho >2$ when more boundary controls are applied. Here, we consider the case of a second, Dirichlet control at $x=0$. In this case, null controllability would be equivalent to the observability estimate
\beq 
\| w_{xxx}(0,*)\|^2_{L^2(0,T)}+\| w_{xxx}(\pi,*)\|^2_{L^2(0,T)}\geq C(\| w(*,0)\|_{X^3}^2+\| w_t(*,0)\|_{X^1}^2), \forall (w^0,w^1)\in X^3\times X^1.
\label{fail}
\eeq 
However,
$$
w_{xxx}(0,t)=C\sum_n n^3(c_n^+e^{\la_n^+(T-t)}+c_n^-e^{\la_n^-(T-t)}),
$$
$$
w_{xxx}(\pi,t)=C\sum_n n^3(-1)^{n+1}(c_n^+e^{\la_n^+(T-t)}+c_n^-e^{\la_n^-(T-t)}),
$$
with $C$ a constant. The resemblance of
$w_{xxx}(0,t)$ to $w_{xxx}(\pi,t)$ makes it unclear to us that the extra term on the left  hand side of \rq{fail} allows us to improve our estimate with the second control.
This is in contrast to the case of internal controls, where for $\al \in (1,3/2)$ we proved null controllability using a 4 dimensional family of controls, see \cite{AEI}.

\end{remark}

\subsection{Proof of Theorem \ref{bcN}}

We propose to study the structurally damped beam equation with Neumann control, which we write formally as
\begin{eqnarray}
u_{tt}+A_N^2 u+\rho \tilde{A}_N^\al u_t& = & 0\label{beamn9}\\
u_{xxx}(0,t)=u_{xxx}(\pi ,t) & = & 0\label{Nbc1n9}\\
u_{x}(0,t) & = & 0\label{Nbc2n9}\\ 
u_{x}(\pi,t) & = & f(t)\label{Nbc3n9}\\
u(x,0)=u^0(x),\ u_t(x,0)& = & u^1(x).\label{Ninitn9}
\end{eqnarray}
Then the adjoint problem is the system \rq{dbeamn}-\rq{dinitn}, and the formulation of the weak solution is given by \rq{weakN}. By the standard duality argument, null-controllability is equivalent to the validity of the 
 associated observability estimate is 
\beq \label{ob3}
\| w_{xx}(\pi,*)\|^2_{L^2(0,T)}\geq C(\| w(*,0)\|_{X^2_{N,0}}^2+\| w_t(*,0)\|_{X^0_{N,0}}^2), \forall (w^0,w^1)\in X^2_{N,0}\times X^0_{N,0}.
\eeq
Note that the function $w(x,t)=\f_0(x)$ is a non-trivial solution to the adjoint system for which   \rq{ob3} would fail, thus the need for 
the quotient spaces $X_{N,0}^p$.
\begin{prop}
\rq{ob3} holds.
\end{prop}
Proof: 
We apply Theorem \ref{win}, choosing $T_0=T$ and $T'=T/2.$
Thus
\begin{eqnarray*}
\int_0^T|w_{xx}(\pi ,t)|^2dt & = &
\int_0^T|\sum_{n=1}(-1)^nn^2(c_n^+e^{\la_n^+(T-t)}+c_n^-e^{\la_n^-(T-t)})|^2dt\nonumber\\
& \geq & {C}\sum_{n=1} |n^2c_n^+e^{\la_n^+T'}|^2
+|n^2c_n^-e^{\la_n^-T'}|^2\nonumber\\
& \geq & {C}\sum_{n=1}n^4e^{-\rho n^{2\al}T/2}(|c_n^+|^2
+|c_n^-|^2),\label{o3n}
\end{eqnarray*}
with $C$ depending on $\rho, \al, T.$ Also, since 
$c_0^+=c_0^-=0$ by hypothesis, we have 
\begin{eqnarray*}
\| w(*,0)\|_{X^2_{N,0}}^2+\| w_t(*,0)\|_{X^0_{N,0}}^2
& = & |\sum_{n=1} |(n^2+\la_n^+)c_n^+e^{\la_n^+T}+
(n^2+\la_n^-)c_n^-e^{\la_n^-T}|^2\\
& \leq & C\sum_{n=1} n^4e^{-\rho n^{2\al}T}(|c_n^+|^2+|c_n^-|^2)\\
& \leq & C\sum_{n=1} n^4e^{-\rho n^{2\al}T/2}(|c_n^+|^2+|c_n^-|^2).
\end{eqnarray*}
Combining  these inequalities, the proposition follows. $\Box$

This completes the proof of Theorem \ref{bcN}.

\subsection{Moment null-control}

For this section, for readability we write 
$X_D^p$ as $X^p$.

We propose to study the structurally damped beam equation with moment control, which we write formally as
\begin{eqnarray}
u_{tt}+A_D^2 u+\rho \A u_t& = & 0\label{beamm}\\
u(0,t)=u(\pi ,t) & = & 0\label{bc1m}\\
u_{xx}(0,t) & = & 0\label{bc2m}\\ 
u_{xx}(\pi,t) & = & g(t)\label{bc3m}\\
u(x,0)=u^0(x),\ u_t(x,0)& = & u^1(x).\label{initm}
\end{eqnarray}
Then the adjoint problem is the system \rq{dbeam'}-\rq{dinit'}, and the definition of the weak solution is given in \rq{wm0}.
Thus null controllability is equivalent to the associated observability estimate: 
\beq \label{ob3m}
\| w_{x}(\pi,*)\|^2_{L^2(0,T)}\geq C(\| w^0\|_{X^1}^2+\| w^1\|_{X^{-1}}^2), \forall (w^0,w^1)\in X^1\times X^{-1}.
\eeq
\begin{prop}\label{p3}
\rq{ob3m} holds.
\end{prop}
Proof: 
We apply Theorem \ref{win}, choosing  $T'=T/2.$
Thus
\begin{eqnarray}
\int_0^T|w_{x}(\pi ,t)|^2dt & = &
\int_0^T|\sum_n\f_n'(\pi)(c_n^+e^{\la_n^+(T-t)}+c_n^-e^{\la_n^-(T-t)})|^2dt\nonumber\\
& \geq & {C}\sum_n |nc_n^+e^{\la_n^+T'}|^2
+|nc_n^-e^{\la_n^-T'}|^2\nonumber\\
& \geq & C\sum_n n^2e^{-\rho n^{2\al}T/2}(|c_n^+|^2
+|c_n^-|^2),\label{o3}
\end{eqnarray}
with $C$ depending on $\rho, \al, T.$
On the other hand, 
\begin{eqnarray}
\| w(*,0)\|_{X^1}^2+\| w_t(*,0)\|_{X^{-1}}^2
& \leq & C\sum_{n=1} n^2e^{-\rho n^{2\al}T}(|c_n^+|^2+|c_n^-|^2).\label{o4}
\end{eqnarray}
Combining  \rq{o3} and \rq{o4}, the proposition follows. 

The following is immediate from Proposition \ref{p3}:
\begin{thm}\label{thmmom}
Let $T>0.$ 
Let $\al \in [0,1)$, $\rho  \leq 2$.
Given $(u^0,u^1)\in X^1\times X^{-1}$, there exists $g\in L^2(0,T)$ such that the solution $u$ to the system   \rq{beamm}-\rq{initm} solves
$$u(x,T)=u_t(x,T)=0,
$$
and there exists a constant $C$ depending only on $\al,\rho , T$ such that
$$\| g\|_{L^2(0,T)}\leq Ce(\| u^0\|_{X^1}+\|u^1\|_{X^{-1}}).$$
\end{thm}

\subsection{Shear control.}
We propose to study the structurally damped beam equation with shear control, which we write formally as
\begin{eqnarray}
u_{tt}+A_N^2 u+\rho \tilde{A}_N^\al u_t& = & 0\label{beams}\\
u_{x}(0,t)=u_{x}(\pi ,t) & = & 0\label{Nbc1s}\\
u_{xxx}(0,t) & = & g(t)\label{Nbc2s}\\ 
u_{x}(\pi,t) & = & 0\label{Nbc3s}\\
u(x,0)=u^0(x),\ u_t(x,0)& = & u^1(x).\label{Ninits}
\end{eqnarray}
Then the adjoint problem is the system \rq{dbeamn}-\rq{dinitn}, and the formulation of the weak solution is given by \rq{weaks}, with observation 
$
w(\pi,t).
$

The following is proven using arguments resembling the ones above, and is left to the reader. 
\begin{thm}\label{thmsh}
Let $T>0.$ 
Let $\al \in [0,1)$, $\rho  < 2$.
Given $(u^0,u^1)\in X^2_N\times X^{0}_N$, there exists $g\in L^2(0,T)$ such that the solution $u$ to the system   \rq{beams}-\rq{Ninits} solves
$$u(x,T)=u_t(x,T)=0,
$$
and there exists a constant $C$ depending only on $\al,\rho ,T$ such that
$$\| g\|_{L^2(0,T)}\leq C(\| u^0\|_{X^2_N}+\|u^1\|_{X_N^{0}}).$$
\end{thm}
\section{Conclusion}

We see our accomplishments in this paper as follows.

1- 
We extended an integration by parts formula for the spectral fractional Laplacian, $(-\Delta )^\al$ with $\al <1$, proven in \cite{E2}for the  Dirichlet Laplacian on rectangular domains, to 
 Neumann boundary conditions on the interval.

 2- For $\al<1$ and any $\rho >0$, we proved well-posedness results for the associated IBVP  with each of the following non-homogeneous boundary values: $u(\pi,t),\ u_x(\pi,t),\ u_{xx}(\pi,t), \ u_{xxx}(\pi,t)$.

3- For $\al < 1$ and $\rho < 2$, we prove null controllability. For other cases, the gap condition on frequencies does not necessarily hold, and failure of the gap condition (except for the case $\la_n^+=\la_n^-$) implies boundary null controllability in finite time will not hold, in contrast to the case of interior control (see, e.g.  \cite{AEI}).

Finally, we note some possible future research. 

A)  One could use the integration by parts formula to discuss well-posedness for other boundary value problems involving $\A$ or $\tilde{A}_N^\al$, such as the wave equation 
$$
u_{tt}-u_{xx}+ \A u_t=0.
$$
This will be considered in future work.

B) 
The integration by parts formula, depending heavily on symmetries, does not readily extend to Laplace-like operators with variable coefficients such as 
$u \mapsto (a(x)u_x)_x$. This problem remains open.

C) The methods of this paper will also apply to beams whose adjoint problem has mixed boundary conditions such as 
$$
w(0,t)=w_{xx}(0,t)=0=w_x(\pi,t)=w_{xxx}(\pi,t).
$$

D) The explicit formula for $J_N$, which is comparable to the corresponding kernel for the ``regional fractional Laplacian" (see \cite{L}) will allow one to compare the eigenvalues of the two operators. This will be considered in future work. 

E) The generalization of Theorems \ref{ibpd}, \ref{ibpn} to non-rectangular domains of higher dimension is not well understood, though see 
\cite{AE1} for partial results. 
\section{Appendix}
\noindent{\bf Proof of Equation \rq{SVN}: }
Assume that $u = \f_j$ is an eigenfunction of the Neumann Laplacian associated to the eigenvalue $\la_j$ . Then, repeating the arguments of (\cite{AD}, Appendix),
$A_N^\al\f_j=\la_j^\al\f_j$, $e^{-tA}\f_j=e^{\la_j}\f_j$, and for all  $x\in (0,\pi)$,
\begin{eqnarray*}
\frac{\Gamma (1-\al)}{\al}A^\al u(x) & =& \int_0^\infty \left (u(x)-e^{-tA}u(x)\right )\frac{dt}{t^{\al+1}}\\
& =& \int_0^\infty \left (u(x)-\int_0^\pi {p}_N(t,x,y)u(y)dy\right )\frac{dt}{t^{\al +1}}\\
&=& \lim_{\ep \ra 0}\int_0^\infty \int_{|y-x|\geq \ep}{p}_N(t,x,y)(u(x)-u(y))dy\frac{dt}{t^{\al +1}}\\
&&+\int_0^\infty u(x)\left( 1-\int_0^\pi {p}_N(t,x,y)dy\right )\frac{dt}{t^{\al+1}}\\
& = & PV\int_0^\pi (u(x) - u(y))J_N(x,y) dy +
\kappa (x)u(x).
\end{eqnarray*}
For Neumann boundary conditions, since $\la_0=0$ and $\f_0=1/\pi$, we have 
$$\kappa (x)=1-\int_0^\pi \int_0^\pi {p}_N(t,x,y)dy=0. $$
For general $u\in X_N^2$, the equations above hold by a density argument; for details the reader is referred to \cite{AD}.

\

\noindent{\bf Proof of Lemma \ref{L2}:} We prove part A, leaving the similar part B to the reader.
  Since a finite union of Bessel sequences is a Bessel sequence, it suffices to prove each that 
$\{ e^{t \la_n^+}: n>M\}$ and $\{e^{t \la_n^-}: n>M\}$ is a Bessel sequence for some positive integer $M$. We give the proof for $\{ e^{t \la_n^+}: n>M\}$, the other proof being similar. For $M$ sufficiently large, $\la^+_n$ is non-real for $n>M$, and   
$$\Im (\la_n^+)= \frac{1}{2}\sqrt{4n^4-\rho^2n^{4\al}}\asymp n^2, \Re (\la_n^+)=\frac{\rho}{2}n^{2\al}
.$$

 For $m\in \bZ,$ let
$
e_m(t)=e^{im(t-T/2)2\pi /T}/\sqrt{T}.
$
Of course, $\{ e_m\}$ is the standard orthonormal basis of $L^2(0,T)$. We define $E$ to be the set of finite linear combinations of $e_m$.
By \cite{He}, it suffices to prove there exists a positive constant $B$ such that for all $f\in E,$
$$
\sum_n \<  e^{t \la_n^+},f\>^2 <B\| f\|^2_{L^2(0,T)}.
$$
 We calculate
$$
\int_0^Te_m(t)e^{t\la_n^+}dt=
\frac{e^{-im\pi}}{\sqrt{T}}\frac{e^{T\la_n^+}-1}{im2\pi/T+\la_n^+}.
$$
Thus for $f(t)=\sum_{m=-N}^Nf_me_m,$ we have
$$
\< f,e^{-t\la_n^+}\> = \frac{2}{\sqrt{T}}\sum_{m=-N}^N
f_me^{-im\pi}\frac{e^{\la_n^+ T }-1}{\rho n^{2\al}+i(m2\pi/T+\sqrt{4n^4-\rho^2n^{4\al}})}.
$$
Then there exists positive constant $B$, independent of $N$, such that 
\begin{eqnarray*}
|\< f,e^{t\la_n^+}\>|^2& \leq & \frac{1}{{T}}\|f\|^2\sum_{|m|\leq N}
|\frac{e^{T\la_n^+ }-1}{\rho n^{2\al}+i(m2\pi/T+\sqrt{4n^4-\rho^2n^{4\al}})}|^2\\
&\leq & \frac{C}{T}\|f\|^2\sum_{|m|\leq N}\frac{1}{(2\pi m/T)^2+n^4}\\
& \leq & B\| f\|^2_{L^2(0,T)}.\Box
\end{eqnarray*}

\subsection{Acknowledgements}

Sergei Ivanov, who was an active contributor in this manuscript, passed away in February.

{The research of S.A. was  supported  in part by the National Science Foundation, grant DMS 2308377, and by the Ministry of Education and Science of the Russian Federations part of the program of the Moscow Center for Fundamental and Applied Mathematics under the Agreement No. 075-15-2025-345.}

\end{document}